\documentclass[11pt]{article}
\usepackage[margin=1in]{geometry}    
\usepackage{gensymb}
\usepackage{graphicx}
\usepackage{amssymb,amsmath}
\usepackage{amstext,amsthm, amssymb,amsfonts}
\usepackage{epstopdf}
\usepackage{multirow}
\usepackage{placeins}
\usepackage{url}
\usepackage{algorithm}
\usepackage{algcompatible}
\usepackage{multicol}

\usepackage[usenames,dvipsnames]{xcolor}
\newcommand{\joel}[1]{\textcolor{black}{#1}}

\newtheorem{theorem}{Theorem}[section]
\newtheorem{lemma}[theorem]{Lemma}

\newtheorem{corollary}[theorem]{Corollary}
\newtheorem{definition}[theorem]{Definition}

\floatname{algorithm}{Algorithm}

\usepackage{multicol}

\title{The connectivity of graphs of graphs with self-loops and a given degree sequence}

\date{}        
\begin{document}
\maketitle

\centerline{\scshape Joel Nishimura}
\medskip
{\footnotesize
 \centerline{School of Mathematical and Natural Sciences}
   \centerline{ Arizona State University, Glendale, AZ 85306-4908, USA}
}

\begin{abstract}
`Double edge swaps' transform one graph into another while preserving the graph's degree sequence, and have thus been used in a number of popular Markov chain Monte Carlo (MCMC) sampling techniques.  However, while double edge-swaps can transform, for any fixed degree sequence, any two graphs inside the classes of simple graphs, multigraphs, and pseudographs, this is not true for graphs which allow self-loops but not multiedges (loopy graphs).  Indeed, we exactly characterize the degree sequences where double edge swaps cannot reach every valid loopy graph and develop an efficient algorithm to determine such degree sequences. The same classification scheme to characterize degree sequences can be used to prove that, for all degree sequences, loopy graphs are connected by a combination of double and triple edge swaps. Thus, we contribute the first MCMC sampler that, asymptotically, uniformly samples loopy graphs with any given sequence.  
\end{abstract}

\section{Introduction}

Understanding what properties of an empirical graph are noteworthy, as opposed to those which are merely the consequence of the degree sequence, is often addressed by comparing the empirical graph with an ensemble of sampled graphs with the same degree sequence \cite{molloy1995critical,newman2001random}.  While uniformly sampling graphs with a fixed degree sequence seems straightforward, it can be surprisingly complex.   
How one samples and the resulting graph statistics are dependent on the space of graphs considered: e.g.~whether self-loops and/or multiedges are considered, and whether graphs with distinct `stub-labelings' are considered unique \cite{carstens2016switching,MRCconfiguration}. 

Graphs which allow self-loops can arise in many disparate applications.  For example, self-loops may represent: an author citing themselves; a protein capable of interacting with itself \cite{Ito10042001,yu2008high}; gene operon self-regulation \cite{shen2002network}; cannibalism in a food web \cite{zander2011food}; users on photo sharing site Flickr linking to themselves; a loop road or cul-de-sac in a road network \cite{demetrescu20069th}; a repeated word in a word adjacency network; traffic flow inside an autonomous system on the Internet \cite{leskovec2007graph},  along with many other possible interpretations.  Considering networks which may include self-loops can be important both because self-loops are often of interest themselves, and because the inclusion of self-loops effectively reduces the number of edges that aren't self-loops, potentially affecting many different network statistics, especially in small networks.  Moreover, while it is commonly thought that self-loops are 
 asymptotically rare \cite{bollobas1980probabilistic}, they only are for particular assumptions on degree sequences, and are more rare in `stub labeled' spaces as thoroughly detailed in  \cite{MRCconfiguration}. In contrast, a so called `vertex-labeled' graph is  more likely to have self-loops, yet techniques for sampling from this space are largely undeveloped {\cite{MRCconfiguration}}. This paper discusses loopy graphs, graphs where each vertex can have at most a single self-loop and edges are either present or absent (i.e.~no multiedges). 

For many different types of graphs, one of the most popular sampling techniques 
is Markov chain Monte Carlo sampling via `double edge swaps' or the more recent `curveball' alteration \cite{strona2014fast,carstens2016curveball,carstens2015Proof}.  Confidence in these MCMC methods rest upon two considerations: whether the stationary distribution of the Markov chain is uniform; and how quickly the chain approximates its stationary distribution as commonly measured by the mixing time. 
{
There have been recent advances in proving polynomial mixing times for chains on simple graphs with constrained degree sequences \cite{cooper2007sampling,greenhill2015switch,mckay1990uniform}, suggesting that future analytical results for loopy graphs are possible.  In the meantime,
}
 there are established numerical methods of accessing the convergence of Markov chains, such as though based on autocorrelation. 
 { In contrast,   }
  determining whether double edge swaps result in a uniform distribution seemingly requires an analytic proof {and such } guarantees are founded on several properties, the most difficult of which is whether an MCMC sampler can sample every possible graph, or equivalently, whether the associated Markov chain is irreducible (equivalently, the associated graph of the Markov chain is strongly connected). For any degree sequence, the following spaces are connected and thus can be sampled using MCMC techniques: simple graphs \cite{zhang2010traversability, bienstock1994degree, berge1962theory, eggleton1981simple, taylor1981constrained}, simple connected graphs \cite{taylor1981constrained,bienstock1994degree}, multigraphs \cite{hakimi1963realizability} and multigraphs with self-loops \cite{eggleton1979graph}.   
Absent from this list is the space of loopy graphs. Indeed, for some degree sequences, the standard MCMC approach applied to the space of loopy graphs cannot sample all possible such graphs.  In this paper, we investigate which degree sequences have disconnected Markov chains, developing an algorithm that can detect this disconnectivity and prove that augmenting the standard MCMC `double edge swap' with `triple edge swaps' guarantees the chain is connected for all degree sequences.
These techniques allow the space of loopy graphs to be used in the study of empirical networks.

\section{The graph of loopy-graphs}

Consider a graph with self-loops $G=(V,E)$, with $n$ vertices in vertex set $V$ and edge set $E$, which may or may not include self-loops: edges of the form $(u,u)$. 
Notice that loopy-graphs include simple graphs as a special case.
 As opposed to multigraphs, edges can appear at most once in $E$.   
For a vertex $u$, we denote the set of adjacent vertices, or `neighbors' of $u$, as $N(u)$, and we refer to {$k_u=|N(u)|$} as the degree of vertex of $u$.   We adopt the convention that each self-loops contributes two to a {vertex's} degree\footnote{{Consider modifications of the method in \cite{carstens2016curveball} if a self-loop contributes only one to a vertex's degree.  }}.  

\begin{figure}[htbp]
  \centering
	\includegraphics[width=\linewidth]{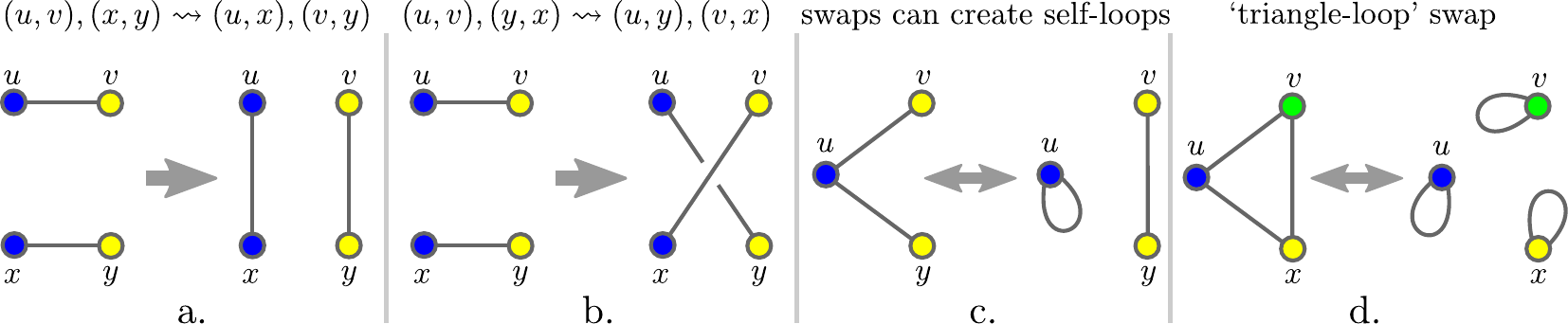} 
	\label{doubleEdgeSwap} %
	\caption{(a, b) For any pair of edges there are two possible double edge swaps. (c)  Swapping adjacent edges creates self-loops, and swaps which involve a single self-loop can remove it. (d) In Theorem \ref{tripleSwaps} we show that including a triple edge swap in addition to double edge swaps leads to a connected graph of loopy graphs $\mathcal{G}_\triangle$.  }
\end{figure}

Transforming one graph into another with the same degree sequence is possible through a double edge swap \cite{petersen1891theorie}, where, as in figure \ref{doubleEdgeSwap}, swapping edges $(u,v)$ and $(x,y)$ replaces those edges with $(u,x)$ and $(v,y)$, a process we denote as $(u,v),(x,y)\leadsto (u,x),(v,y)$.  Repeatedly performing double edge swaps (resampling the current graph whenever a proposed swap would create a multiedge) is the basis of MCMC samplers of loopy graphs.  Conceptually, repeatedly performing double edge swaps is a random walk on a graph whose vertices are loopy graphs, with the same prescribed degree sequence.  
Specifically, 
{
this allows us to develop a notion of a `graph of loopy graphs':
\begin{definition}[graph of loopy graphs, $\mathcal{G}(\{k_u\})$]
\label{def:GoG}
For a given degree sequence $\{k_u\}$ let $\mathcal{G}(\{k_u\})=\{\mathcal{V},\mathcal{E} \}$ be the graph of loopy graphs under double edge-swaps, where vertex set $\mathcal{V}$ contains each possible loopy graph with degree sequence $\{k_u\}$ and edge set $\mathcal{E}$ contains $(G_i,G_j)$ if and only if there is a single double edge swap that takes $G_i$ to $G_j$.
\end{definition}
}
\noindent Showing that a MCMC sampler can sample from graphs with degree sequence $\{k_u\}$ requires showing that $\mathcal{G}(\{k_u\})$ is connected, otherwise random walks will not be able to reach all possible loopy-graphs.

In fact though, the space of loopy-graphs is not connected under double edge-swaps for every possible degree sequence. In  section \ref{counterExamples}, we introduce two classes of graphs, $Q_1$ and $Q_2$ such that if $\mathcal{G}$ contains a loopy-graph $G\in Q_1\cup Q_2$, then $\mathcal{G}$ is disconnected.  Both classes $Q_1$ and $Q_2$ require high degree nodes, and in section \ref{detecting} we discuss tests to determine whether a given degree sequence can create a graph in $Q_1$ or $Q_2$.  

Moreover,  $Q_1$ and $Q_2$ exactly characterize the graphs which cause $\mathcal{G}$ to be disconnected, as shown in the first main theorem, Theorem \ref{Q1Q2Exact} in section \ref{ghat}.  
In contrast, any two graphs not in $Q_1$ or $Q_2$ are connected with each other, and this will be established by studying special maximal elements of $\mathcal{G}$.

 The general outline is as follows: from any graph $G_i$, let $\hat{G}_i$ be the graph {in the same connected component of $\mathcal{G}$ as} $G_i$ with the maximum number of self-loops ({see definition \ref{def:Ghat} for} for additional technical {requirements}), as in Figure \ref{sampleGraphs}.

 Next, we utilize the following classification of vertices inside a graph:
 \begin{definition}[$V^k$]
 \label{def:Vk}
For graph $G$, let $V^0$ be the set of all vertices in $G$ that lack self-loops   (i.e. $V^0 = \{u|(u,u)\not \in E  \}$) . Let $V^k$ be the set of all vertices with a shortest path distance of $k$ from any vertex in $ V^0$. Let $V^\infty$ be those vertices which are disconnected to any vertex in $V^0$.
\end{definition}
 
\noindent Based on the largest clique $K^0\subseteq V^0$ {(see definition \ref{def:Kzero})}, we classify the structure of $\hat{G}_i$ as one of five different types (see definition \ref{def:Ghatd} for additional technical requirements), (four types are displayed in Figure \ref{GdForm}), two of which ($\hat{G^3}$ and $\hat{G}^d$, $d>3$) belong to $Q_1$ and $Q_2$.  

\begin{figure}[htbp]
  \centering
	\includegraphics[width=.95\linewidth]{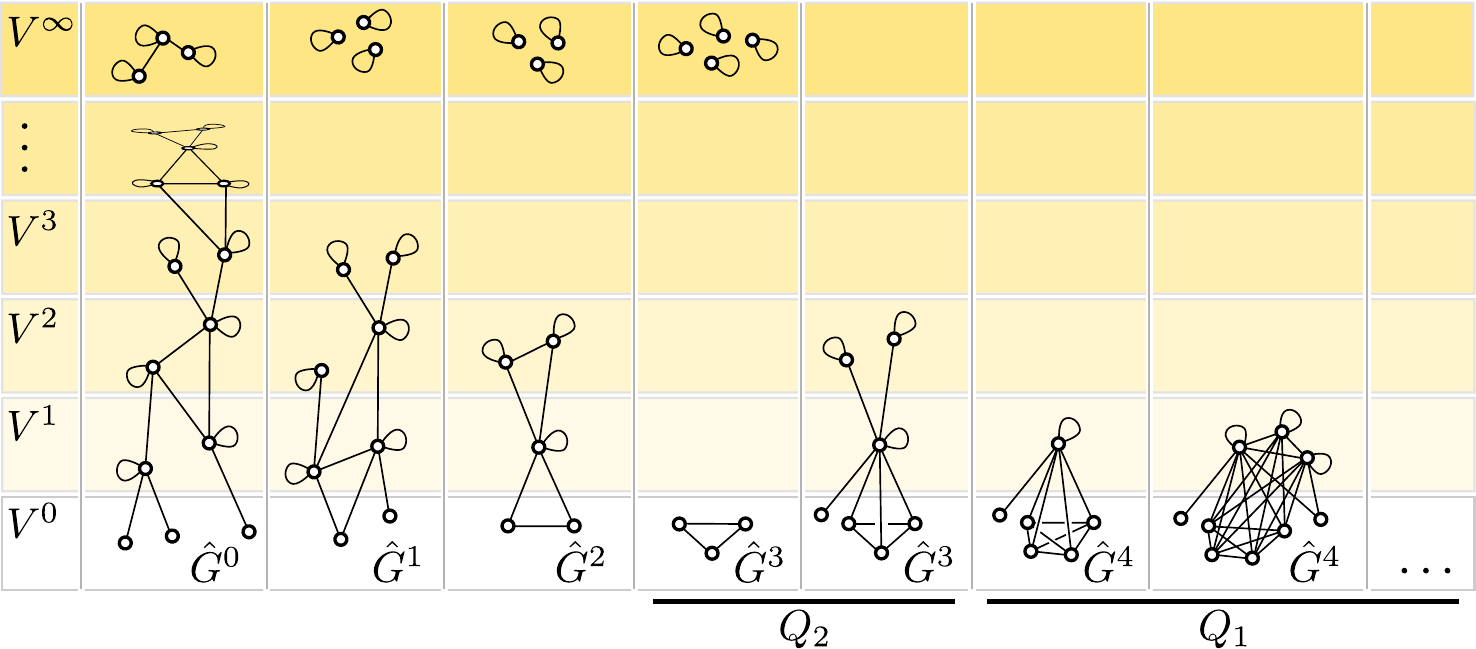} 
	\label{sampleGraphs} 
	\caption{Each of the above graphs has the maximal number of self-loops accessible through double edge swaps. Vertices are categorized by their distance to a vertex without a self-loop and graphs are labeled by the size of the largest clique in $V^0$ {such that a graph $G^d$ has an $d$-clique in $V^0$ (see definition \ref{def:Ghatd})}. Graphs with cliques of size $3$ or greater in $V^0$ are of class $Q_1$ or $Q_2$, as shown by theorems \ref{G3inQ} {and \ref{GdinQ}}.   }
\end{figure}

The categorization of possible structures of $\hat{G}$ suggests Algorithm \ref{alg1} which determines whether a degree sequence has a connected or disconnected $\mathcal{G}$.   
Another consequence of Theorem \ref{Q1Q2Exact} is that any degree sequence $\{k_u \}$ that is disconnected, is disconnected because there are graphs with triangles which cannot be changed into self-loops, which naturally suggests Theorem \ref{tripleSwaps}, which states that the space of graphs with self-loops is connected under the combination of double and triple edge swaps.
 Based on this theorem we suggest a MCMC approach that uniformly samples graphs with self-loops and a fixed degree sequence.
 
\section{Degree sequences with disconnected $\mathcal{G}$} \label{counterExamples}

First we consider a simple disconnected case, which establishes that for some degree sequences $\mathcal{G}$ is not connected.

\subsection{Cycles and cliques}

The simplest example of a degree sequence that is not connected is $\{2,2,2\}$, which can be wired either as a triangle, or as $3$ self-loops.  Since there are no valid double edge swaps of either the triangle graph or $3$ self-loops (all swaps would create multiedges) the space is disconnected.  The disconnectivity of $\{2,2,2\}$ can be extended in two ways, to larger cycles and to larger cliques. The degree sequence of a cycle, $\{2,2,...,2\}$ clearly has a disconnected space, since a graph composed only of nodes with self-loops has no valid double edge swaps.  Similarly a clique with additional self-loops on no more than $n-3$ vertices has alternate configurations, but lacks any valid double edge swaps, implying that the degree sequence: $\{n+1,...,n+1,n-1,...,n-1,n-1  \}$ is also disconnected.  

As a useful exercise, we consider the structure of $\mathcal{G}(\{2,2,....,2\})$ in more detail. Any graph with degree sequence  $\{2,2,....,2\}$ is composed of isolated self-loops and cycles of length at least $3$. Further, any valid double edge swap either:
\begin{enumerate}
\item creates a self-loop and reduces a $k$ cycle, $k\ge 4$, to a $k-1$ cycle (swapping adjacent edges);
\item combines a self-loop with a $k$ cycle to create a $k+1$ cycle (swapping a self-loop and an edge in a cycle);
\item merges two cycles into a larger cycle (swapping edges in separate cycles);
\item cuts a cycle into two smaller cycles, each with length at least $3$ (non-adjacent edges in the same cycle);
\item swaps two edges in the same cycle without changing its length (non-adjacent edges in the same cycle).
\end{enumerate}

If double edge swaps are augmented with a triple edge swap that takes a triangle to three self-loops (and another triple edge swap that does the reverse), then it is clear that every graph in the space can be taken to the graph made entirely of self-loops (and thus $\mathcal{G}$ is connected) via the following procedure:
\begin{enumerate}
\item by swapping edges in different cycles, combine all cycles into a single long cycle;
\item from the graph's one cycle, swap adjacent edges to create self-loops until the single cycle has length $3$;
\item use a triple edge swap to replace the only length $3$ cycle with $3$ self-loops.
\end{enumerate}

\subsection{Other disconnected graphs}

These disconnected examples will be generalized into two classes of graphs  $Q_1$ and $Q_2$, displayed in Figure \ref{FigQ1Q2}, which generalize the problems with the clique and the cycle respectively. In section \ref{ghat} we show that $Q_1$ and $Q_2$ describe all disconnected graphs.

\begin{figure}[htbp]
  \centering
	\includegraphics[width=.95\linewidth]{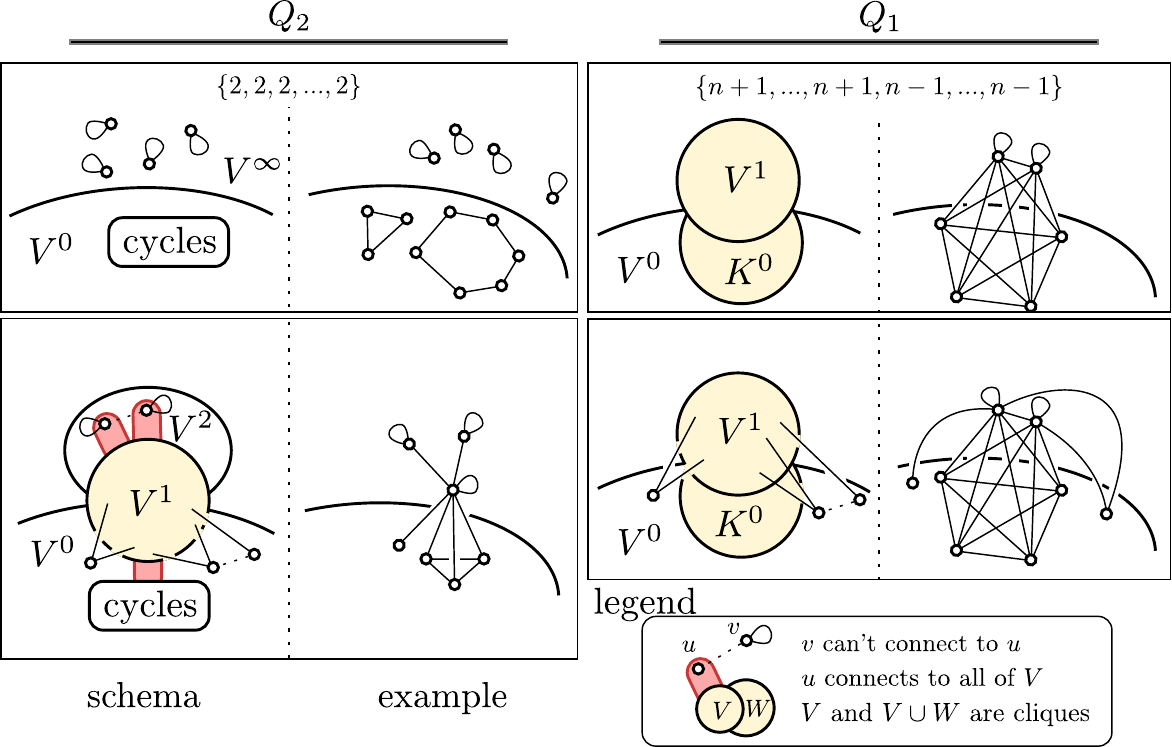}  
	\caption{Both the degree sequence $\{n+1,...,n+1,n-1,...,n-1 \}$ and $\{2,2,...,2\}$ have a disconnected $\mathcal{G}$ whose disconnectivity can be generalized to classes $Q_1$ and $Q_2$. The schematic for $Q_2$ includes  $\{2,2,...,2\}$ as a special case if when $V^1$ is empty, $V^2$ is relabeled as $V^\infty$. \label{FigQ1Q2} }
\end{figure}

   \begin{definition}[$Q_1$]
A graph $G$ is of class $Q_1$ when the following conditions are true of $G$:
\begin{enumerate}
\item There exists a clique $K^0$ in $V^0$ with $|K^0|\ge 4$ (recall: $V^0$ is the set of nodes without self-loops) 
\item For any $u\in V^0$, either $u$ has no neighbors in $V^0$ or $u$ is in the clique $K^0$,
\item $V^1\cup K^0$ is a clique,
\item $V^2 = V^\infty = \emptyset$.
\end{enumerate}
\end{definition}
We will later show that all $\hat{G}^d$, $d>3$ are of class $Q_1$.  The important feature of $Q_1$ is that it is closed under any double edge swap.

 \begin{lemma}
 For any two graphs $G_1$ and $G_2$ connected via a double edge swap, if $G_1\in Q_1$ then $G_2\in Q_1$. 
 \label{Q1closed}
 \end{lemma}

      \begin{proof}
The structure of $Q_1$ implies that all edges have at least one endpoint in  $V^1\cup K^0$.   Since $V^1\cup K^0$ is a clique, there are thus no valid swaps involving any edge in $V^1\cup K^0$ as any such swap would create a multiedge. Similarly, a swap between a self-loop in $V^1$ and an edge from $V^0$ to $V^1$ would also create a multiedge. The only possible swaps are between two edges $(u,v)$ and $(x,y)$ where $u,x\in V^0\setminus K^0$ and $v,y\in V^1$.  Notice that swap $(u,v)(y,x)\leadsto (u,x),(y,v)$ is precluded by the presence of edge $(v,y)\in V^1$, while swap $(u,v)(x,y)\leadsto (u,y),(x,v)$ does not create a new edge in $V^0$, alter the fact that $V^1\cup K^0$ is a clique or create a vertex in $V^2$ or $V^\infty$.  Thus $Q_1$ is closed under edge swaps.   
 \end{proof}

We will later show that all $\hat{G}^3$ are of class $Q_1$.
 While $Q_1$ includes cliques as a special case, a similar structure, $Q_2$ generalizes the problems associated with cycles and degree sequences $\{2,2,2,...,2 \}$. 
 
 \begin{definition}[$Q_2$]
A graph $G$ is of class $Q_2$ when the following conditions are true of $G$:
\begin{enumerate}
\item There are at least three vertices in $V^0$ and there exists $u\in V^0$ such that $|N(u)\cap V^0| = 2$ 
\item For any $u\in V^0$, either $u$ has no neighbors in $V^0$, or $u$ has exactly two neighbors in $V^0$ and is adjacent to all of $V^1$.
\item $V^1$ is a clique
\item For any $u\in V^2$, $N(u)=V^1$,
\item $V^3 = \emptyset$,
\item Either $V^\infty$ is empty or both $V^1$ is empty and $k_u = 2$ for $u\in V^\infty$. 
\end{enumerate}
\end{definition}
Implicit in the definition of $Q_2$ is that there is a cycle in $V^0$ of length at least $3$. Similarly to $Q_1$, $Q_2$ is also closed under double edge swaps.
 \begin{lemma}
  For any two graphs $G_1$ and $G_2$ connected via a double edge swap, if $G_1\in Q_2$ then $G_2\in Q_2$.  \label{Q2closed}
 \end{lemma}
  \begin{proof}
 If $V^\infty$ is non-empty then condition $6$ of $Q_2$ implies that $V^1 = \emptyset$ and that $k_u = 2$ for all $u\in V^\infty$.  Since $V^1 = \emptyset$ then condition $2$ implies that all non-isolated vertices in $V^0$ have degree $2$.  Thus, the degree sequence of non-isolated nodes is $\{2,2,...,2\}$, and this scenario was fully described earlier. 
 
If $V^1\ne \emptyset$, 
a quick check reveals that the only edge swaps that are possible (all others would require multiedges) involve swaps between two edges in $V^0$, swaps between an edge in $V^0$ and a self-loop in $V^2$, and swaps between two edges joining $V^0$ to $V^1$. However, each of these three swaps preserves the properties of $Q_2$: swaps between the two edges in $V^0$ rearrange the cycle structure of $V^0$ and potentially move a node from $V^0$ to $V^2$, but this preserves the properties of $Q_2$; swaps between a self-loop in $V^2$ and an edge in $V^0$ move a node from $V^2$ to $V^0$, reversing the previous swap; For edges $(u,v)$ and $(x,y)$, $u,x\in V^0$ and $v,y\in V^1$ swap $(u,v)(y,x)\leadsto (u,x),(y,v)$ is precluded by the presence of edge $(v,y)\in V^1$, while swap $(u,v)(x,y)\leadsto (u,y),(x,v)$ is only possible if both $|N(u)\cap V^0|= |N(x)\cap V^0| = 0$ and such a swap does not affect the properties of $Q_2$.  
 \end{proof}
 
 This implies the first half of Theorem \ref{Q1Q2Exact}:
 
 \begin{corollary}
 Any $\mathcal{G}$ which contains a graph in $Q_1$ or $Q_2$ is disconnected. \label{GQ1Q2}
 \end{corollary}
 \begin{proof}
 All graphs in $Q_1$ and $Q_2$ contain a closed cycle of length at least 3  in $V^0$. For a graph $G\in \mathcal{G}$ and $G\in Q_1 \cup Q_2$ let $C$ be all the cycles in $V^0$.  Deleting each edge in $C$ and placing a self-loop at each node in $C$ preserves the degree  sequence and thus creates a graph $H\in \mathcal{G}$, but $H$ does not satisfy the first criterion of either $Q_1$ or $Q_2$.  By  lemmas \ref{Q1closed} and \ref{Q2closed}, $Q_1$ and $Q_2$ are closed under double edge swaps and thus $G$ is not connected to $H$.
 \end{proof}
 
Thus, we have generalized the { the disconnectivity of $\{2,2,2\}$ in two directions, first to cliques, and then to $Q_1$, and second to cycles and then to $Q_2$.  
} 
As shown in the next section, $Q_1$ and $Q_2$ exactly characterize all disconnected $\mathcal{G}$.


\section{Categorizing the components of $\mathcal{G}$} \label{ghat}

{First consider the following definitions.} For any graph $G_i\in \mathcal{G}$,  let {$\mathcal{V}_*(G_i)$} be the graphs {in the same connected component of $\mathcal{G}$ as } $G_i$ with the maximum number of self-loops.
\begin{definition}[ $\hat{G}_i$] 
\label{def:Ghat}
 For an initial graph $G_i$, of the graphs in $ \mathcal{V}_*(G_i)$, let $\hat{G}_i\in \mathcal{V}_*(G_i)$ be a graph with the maximum number of edges inside $V^0$.
\end{definition}
\noindent   {To emphasize, $\hat{G}_i$ has the maximum number of self-loops of any graph path-connected to $G_i$, and secondly, has as many edges inside $V^0$ as any other path-connected graph with the same number of self-loops. This definition implies that there are no sequence of edge swaps which can net increase the number of self-loops in a graph $\hat{G}_i$.} While $\hat{G}_i$ has at least as many self-loops as any other graph connected to $G_i$, if $\mathcal{G}$ is not connected, $\hat{G}_i$ may not have the maximum number of self-loops possible.

Before we formalize the meaning of a graph with the maximum number of self-loops,  let $\{ \bar{k}_u \}$ denote the `simplified degree sequence' of a graph $G$:
\begin{definition}[simplified degree sequence $ \{ \bar{k}_u \}$]
\label{def:simplifiedegreesequence}
For a graph $G$, with degree sequence $\{k_u\}$, let $ \{ \bar{k}_u \}$ be the simplified degree sequence, where $\bar{k}_u = k_u$ if $(u,u)\not \in E$ and $\bar{k}_u = k_u-2$ if $(u,u)\in E$.
\end{definition}
\noindent The simplified degree sequence of a graph is the new degree sequence that results from deleting all self-loops in that graph. For the following, assume the degree sequence $\{k_u\}$ is in decreasing order.

\begin{definition}[$m$-loopy graph and $m^*$-loopy graph]
A graph $G$ is $m$-loopy if $G$ has $m$ self-loops on vertices $i\le m$. We denote an $m$-loopy graph as $m^*$-loopy if it has no fewer self-loops than any other $m$-loopy graph in $\mathcal{G}$. 
\end{definition}
\noindent 
We will show in corollary \ref{lem:m*exists} that every $\mathcal{G}$ contains an $m^*$-loopy graph, and that each such graph contains the maximum number of self-loops possible (not just the maximum of $m$-loopy graphs). 
Notice that the question of whether a $m$-loopy graph is $m^*$-loopy is equivalent to whether there is a $m_1>m$ such that $\bar{k}_u = k_u-2$ for $u\le m_1 $ and $\bar{k}_u = k_u$ for $u>m_1$ is a simple-graphical degree sequence (i.e. if there exists a simple graph with that degree sequences).  

Determining the cases where $\hat{G}$ is  $m^*$-loopy will be critical in the categorization of different possible $\hat{G}$ for the following reason.  

\begin{lemma}
For $G_1, G_2 \in \mathcal{V}$, if both $G_1$ and $G_2$ are $m^*$-loopy then $G_1$ is connected to $G_2$. \label{m*loopy}
\end{lemma}
\begin{proof}
Since both $G_1$ and $G_2$ are $m^*$-loopy they have self-loops at the same vertices and the same simplified degree sequences and thus, by the connectivity of simple graphs \cite{taylor1981constrained}, $G_1$ and $G_2$ are connected.
\end{proof}

In some degree sequences, a  graph is obviously $m^*$-loopy because all nodes with degree at least two have self-loops.  It is not always as straightforward though.  
For example, the degree sequences $\{4,4,2 \}$ and $\{6,6,5,3,3,3,2  \}$ have no configurations where all vertices have self-loops, as $\{2,2,0\}$ and $\{4,4,3,1,1,1,0 \}$ are not simple-graphical degree sequences.  Instead, these graphs have valid configurations where all but the vertex with degree $2$ has self-loops.  

For any degree sequence there are thus two possibilities, either all graphs in $\mathcal{G}$ are connected to $m^*$-loopy graphs and $\mathcal{G}$ is connected, or there exists some graph not connected to any $m^*$-loopy graph and $\mathcal{G}$ is not connected.    

Understanding the possible forms of $m^*$-loopy graphs will comprise the majority of the remaining effort, but the simplest case may also be the most common case. 

{
\begin{lemma}
\label{lem:01mstar}
 For any $\hat{G_i}$ where $V^0$ contains only vertices of degree $0$ and $1$, $\hat{G_i}$ is $m^*$-loopy.
\end{lemma}
\begin{proof}
If $V^0$ has only vertices of degree $0$ and $1$ then all other vertices have self-loops.  Since vertices of degree $0$ and $1$ can not have self-loops $\hat{G}$ is $m^*$-loopy.
\end{proof}
}

We now turn our attention to the much more complicated scenarios where there exists some $u\in V^0$ with $k_u\ge 2$.  

{In order to further classify the different possible structures of different $\hat{G}_i$ consider the following definitions:
\begin{definition}[$K^0$]
\label{def:Kzero}
In a graph $\hat{G}_i$, Let $K^0$ refer to any of the largest sized cliques inside $V^0$.
\end{definition}
}
{As we will show, if $|K^0|>2$ then $V^0$ houses only a single clique, in which case $K^0$ is the unique clique.  We now classify $\hat{G}_i$ based upon the cliques inside $V^0$:  
}
{
\begin{definition}[$\hat{G}^d$]
\label{def:Ghatd}
$\hat{G}^d$ is a graph $\hat{G}_i$ with $|K^0 |=d$ where some $u\in V^0$ has $k_u\ge2$. 
\end{definition}
Following the proof of lemma \ref{exchange} we will assume WLOG that in a graph $\hat{G}^d$, $k_u\le k_x$ for any $u\in V^0$ and $x\not \in V^0$.
}

The critical lemmas to prove will be Lemmas \ref{exchange}, and \ref{masterLoopy}.  Lemma \ref{exchange} states that there exists a sequence of double edge swaps which can exchange any vertex in $V_0$ with any other vertex of equal or lower degree.  Thus any $\hat{G}_i$ is connected to another graph $\hat{G}_j$ where $V^0_j$ contains only the smallest degrees.  Building on this, Lemma \ref{masterLoopy} states that a graph $\hat{G}^d$, $d\le 2$, is $m^*$-loopy.
Thus, by lemma \ref{m*loopy} only degree sequences that can wire  a $\hat{G}^d$, $d\ge 3$ can be disconnected and, as will be shown in theorem \ref{Q1Q2Exact}, {any graph $\hat{G}^d$, $d\ge3$ implies disconnectivity because} a graph $\hat{G}^3\in Q_2$ and $\hat{G}^d\in Q_1$ for $d>3$ and $Q_2$ and $Q_1$ are closed.

\begin{figure}[htbp]
  \centering
	\includegraphics[width=.75\linewidth]{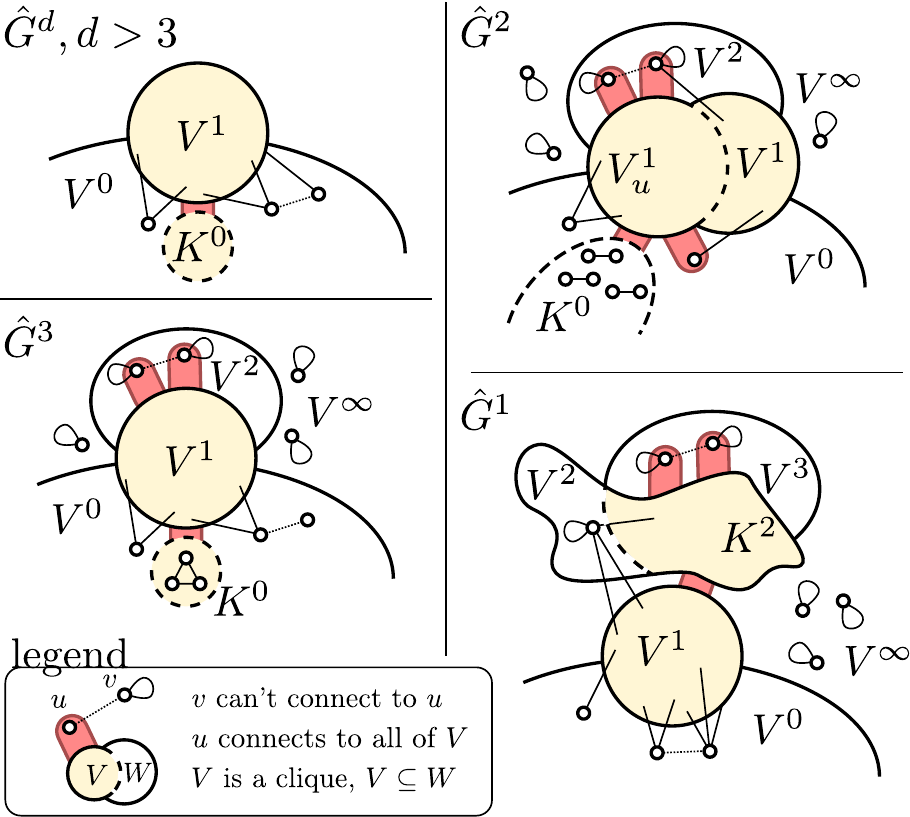} \label{GdForm} 
	\caption{The size of $K^0$ imposes strict requirements on the possible structure of $\hat{G}^d$.  When $K^0$ is a clique larger than $3$ vertices ($d>3$), $\hat{G}^d$ is simply a clique of vertices some with self-loops, some without, and a number of vertices with connections only into members of the clique which have self-loops.  Note: in $\hat{G}^3$ either $V^\infty$ or $V^1$ must be empty. }
\end{figure}

Before proving lemmas \ref{exchange} or \ref{masterLoopy} we first construct some general purpose lemmas. We begin with some investigations into restrictions on the sets $V^k$ for all $\hat{G}^d$.  Consider the following definition: 
\begin{definition}[Open Wedge $xuv$]
{
If $(u,v)\in E$, $(u,x)\in E$ and $(v,x)\not \in E$ then there exists open wedge $xuv$.
}
\end{definition}
{
\begin{lemma}
\label{noOpenWedge}
A graph $\hat{G}_i$ does not have an open wedge $xuv$ where $u\in V^0$.
\end{lemma}
\begin{proof}
Suppose not, then the swap $(u,x),(u,v)\leadsto (u,u),(x,v)$ is possible and creates a self-loop, violating the assumption that $\hat{G}_i$ had the maximum number of self-loops in its connected component of $\mathcal{G}$.
\end{proof}
}

\begin{lemma}
{For a graph $\hat{G}_i$,} any $u\in V^0$, and $v\in N(u)$, if $u$ connects to a vertex $x$, so does $v$.\label{K0_edges}
\end{lemma}
\begin{proof}
If not, then there exists open wedge $xuv$, {contradicting lemma \ref{noOpenWedge}}.
\end{proof}

Lemma \ref{K0_edges} implies that any subgraph of $V^0$ are cliques plus isolated vertices, as is the subgraph on vertices $u\cup N(u)$ for $u\in V^0$.

\begin{lemma}
{For a graph $\hat{G}_i$, if} there exists disjoint $(u,v)$ and $(x,y)$ both in $V^0$, $k_u \ge 2$ then $N(u)\cap V^1=N(x)\cap V^1$
\end{lemma}
\begin{proof}
Suppose first that $N(u)\cap V^1 \not \subseteq N(x)\cap V^1$, then there exists $w\in N(u)\cap V^1$ such that $w\not \in N(x)$. If $(u,x)$ exists, then by lemma \ref{K0_edges} $x$ must be connected to $w$, a contradiction.  Thus $(u,x)$ isn't present, and similarly, lemma \ref{K0_edges} also implies that $(u,y)$, $(v,x)$ and $(v,y)$ aren't present.  
Swap $(u,v),(x,y)\leadsto(u,x),(v,y)$ is thus possible but creates open wedge $wux$ contradicting that $\hat{G}$ has the maximal number of self-loops.

If instead $N(x)\cap V^1 \not \subseteq N(u)\cap V^1$ then $k_x\ge 2$ and the above argument holds.
\end{proof}

\begin{lemma}
{For a graph $\hat{G}_i$, if} there exists $u\in V^0$, $k_u\ge2$ then {any edge which is not a self-loop contains} a vertex in $V^0$, $V^1$ or $V^2$. \label{noV3V3}
\end{lemma}
\begin{proof}
Suppose to the contrary that there exists $(x,y)$ with neither $x$ nor $y$ in $V^0$, $V^1$ or $V^2$.  For $v,w \in N(u)$, notice $(v,w)$ must exist, otherwise there exists an open wedge $vuw$, swap $(v,w),(x,y)\leadsto(x,v),(y,w)$ is thus valid, but creates an open wedge at $u$.  
\end{proof}

This implies that $V^d$ is empty for all finite $d>3$, and $V^3$ contains no edges asside from self-loops. This also implies that the set of vertices disconnected from $V^0$ can only contain isolated self-loops.

\begin{lemma}
For $\hat{G}^d$ with $d\ge 3$, $V^1 \subset N(u)$ for all $u\in K^0$. \label{K0V1Conn}
\end{lemma}
\begin{proof}
Suppose not, then there exists $x\in V^1${, $u\in K^0$ such that $(u,x)\not \in E$.  By Lemma \ref{K0_edges} it must be that $x$ does not neighbor any vertices in} $K^0$.  Since $x\in V^1$ there exists $y\in V^0\setminus K^0$ along with edge $(x,y)$.  Let $u,v,w\in K^0$. Lemma \ref{K0_edges} implies that since $y \not \in K^0$ then $y$ {does not neighbor any vertices} of $K^0$.  Swapping $(u,v),(x,y)\leadsto(u,x),(v,y)$ creates an open wedge $yvw$.  
\end{proof}

\begin{lemma}
{In a graph $\hat{G}_i$, }if there exists $u\in V^0$, $k_u\ge2$, then $V^1$ is a clique. \label{V1clique}

\end{lemma}
\begin{proof}
If $|V^1|=1$, then $V^1$ is trivially a clique. If $|V^1|\ge2$, then suppose to the contrary that there exists $x,y\in V^1$ such that $(x,y)\not \in E$.   Consider the two possible cases:  
\begin{enumerate}
\item There exists some $u\in V^0$ such that $x,y \in N(u)$. In this case, if $(x,y)\not \in E$ then there is an open wedge $xuy$.  Thus $(x,y)\in E$.
\item There exists $u,v\in V^0$, such that $(u,x)$ and $(v,y)$ are in $E$ but $(u,y)$ and $(v,x)$ are not. If $(u,v)\in E$ then there exists an open wedge $xuv$.  Thus both $(x,y)$ and $(u,v)$ are not in $E$ and the swap $(u,x),(v,y)\leadsto(x,y),(u,v)$ is valid, but produces a graph with one additional edge in $V^0$, contradicting that $\hat{G}$ has the maximum number of edges in $V^0$.

\end{enumerate}
\end{proof}

\begin{lemma}
In $\hat{G}_i$, for any vertex $x\not \in V^0$, and any vertex $u\in V^0$, if {$k_x \le k_u$} then there exists a sequence of swaps that exchanges $x$ for $u$ in $V^0$ {without decreasing the number of edges in $V^0$}. \label{exchange}
\end{lemma}
\begin{proof}
First, we consider the case where $x\in V^1$ {and show that $k_x\ge k_u$}.  For $d\in \{1,2\}$, since $V^1$ is a clique, each $x\in V^1$ contains a self-loop and any $u\in V^0$ contains at most a single neighbor not in $V^1$ then $k_x \ge k_u$. 
For $d\ge 3$,  lemmas \ref{K0V1Conn} and \ref{V1clique} imply that $K^0\cup V^1$ is a clique, and lemma \ref{K0_edges} implies that all subgraphs of $V^0$ are cliques, then $k_x \ge k_u+2$ for any $x\in V^1$ and $u\in V^0$.  

For $x\in V^m$, $m\ge 2$ suppose that there exists $x$ with degree less than $u$. Consider two cases, first that $N(u)\subseteq V^1$ and second that there exists edge $(u,z)\in V^0$.

\begin{enumerate}
\item $N(u)\subseteq V^1$:  Since $(x,x)$ contributes $2$ to $x$'s degree, $N(u)\subseteq V^1$ and $k_x\le k_u$ then there exists $v,w\in N(u)$ and $v,w \not \in N(x)$.  In such a case, notice that swap $(x,x),(v,w)\leadsto (x,v),(x,w)$ and subsequent swap $(u,v),(u,w)\leadsto(u,u),(v,w)$ exchanges $x$ for $u$ in $V^0$.

\item There exists edge $(u,z)\in V^0$:  First, swap $(x,x),(u,z)\leadsto(x,u),(x,z)$.  Since $k_x\le k_u$ and $x$ is connected to $z$ while $z\not \in N(u)$ then there must be some $y\in N(u)$ but $y\not \in N(x)$.  Thus there exists open wedge $xuy$ and swap $(x,u),(u,y)\leadsto (u,u),(x,y)$ exchanges  $x$ for $u$ in $V^0$. {Since it was originally the case that $z\in N(u)$, then $N(z)\setminus u = N(u)\setminus z$ by lemma \ref{K0_edges}. Now that $x\in N(z)$ then it must be that $N(x)\setminus z = N(z)\setminus x$ or else, as in lemma \ref{K0_edges} there would exist an open wedge, leading to a graph with additional self-loops.  That $N(x)\setminus z = N(z)\setminus x$ implies that the new graph has the same number of edges in $V^0$.   }

\end{enumerate} 
\end{proof}

{
\begin{lemma}
\label{lem:GhatMsimple}
Every $G_i$ is connected in $\mathcal{G}$ to an $m$-loopy $\hat{G}_i$. 
\end{lemma}
\begin{proof}
Suppose $\hat{G}_i$ is not $m$-loopy. Lemma \ref{exchange} implies that there exists a series of swaps that preserve the number of self-loops, but place all self-loops on the first $m$ largest degree vertices.  Since these swaps do not decrease the number of edges in $V^0$, the resulting graph has at least as many self-loops and as many edges in $V^0$ as $\hat{G}_i$ and is thus also a valid $\hat{G}_i$.
\end{proof}
}

{
\begin{corollary}
\label{lem:m*exists}
Every nonempty $\mathcal{V}$ contains an $m^*$-loopy graph, and every $m^*$-loopy graph contains the maximum number of self-loops.  
\end{corollary}
\begin{proof}
Let $G_j$ be such that $\hat{G}_j$ has the maximum number of self-loops of all graphs in $\mathcal{V}$.  By lemma \ref{lem:GhatMsimple}, one of these $\hat{G}_j$ is $m$-loopy, and since it has the maximum number of self-loops possible, it must also be $m^*$-loopy 
\end{proof}
}

Based on lemma \ref{lem:GhatMsimple} we will assume WLOG that $\hat{G}^d$ refers to a $m$-loopy graph. It thus remains to show that $\hat{G}^d$ is $m^*$-loopy for $d \in \{1,2\}$ (lemma \ref{masterLoopy}) and to further restrict the possible structures when $d\ge 3$.  


\subsection{The structure of $\hat{G}^1$}

\begin{figure}[htbp]
  \centering
		\includegraphics[width=.45\linewidth]{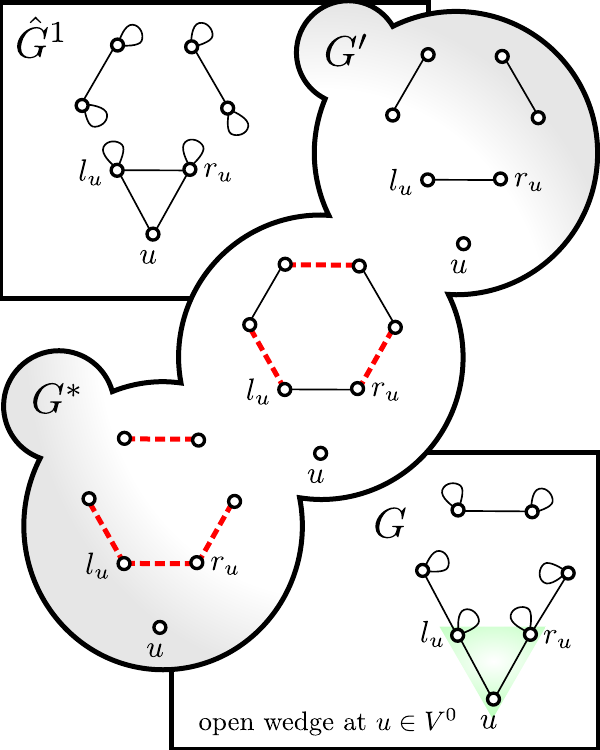}
	\caption{If there exists a graph with more self-loops than a graph $  \hat{G}^1$   then an alternating cycle argument can show that there exists a graph $G$, with the same simplified degree sequence as $\hat{G}$ but with an open wedge at $u\in V^0$.  \label{altpath}}
\end{figure}

\begin{lemma}
Every {m-loopy} $\hat{G}^1$, is $m^*$-loopy. \label{k1graph}
\end{lemma}
\begin{proof}
Suppose not, then by lemma \ref{lem:GhatMsimple} there exists $m$-loopy $\hat{G}^1$ which is not $m^*$-loopy. 
{
Let $\hat{G}_j$ be the $m^*$-loopy graph guaranteed by corollary \ref{lem:m*exists}, and let $S$ be the vertices on which $\hat{G}_j$ but not $\hat{G}^1$ has self-loops (namely, the sequential indicies $S=\{m+1, \cdots m^* \}$).  Next, let $G^*=\{V^*,E^*\}$ be the simple graph attained from $\hat{G}_j$ by deleting all self-loops, as in Figure \ref{altpath}.  
}


For each $u\in S$ there must be at least two vertices $l_u,r_u\in N(u)$ where $l_u, r_u \not \in N^*(u)$. Let $B = \bigcup_{u\in S} \{ l_u \cup r_u \}  $ and let $G' = \hat{G}^1$ except without self-loops and edges $(u,l_u)$ and $(u,r_u)$ for each $u\in S$.  Notice that $G'$ and $G^*$ have the same degree sequence, except at vertices $B$, where those in $G^*$ have a greater degree. 

Let $\Omega' = E'   \setminus E^* $ be the edges in $G'$ not in $G^*$ and let $\Omega^* = E^*\setminus E' $ be the edges in $G^*$ not in $G'$.  Now consider the edge disjoint cycles and paths which alternate between edges in $\Omega'$ and $\Omega^*$. Since the degrees of all vertices in $V\setminus B$ is the same in $G'$ and $G^*$,  there exists a decomposition that consists entirely of alternating cycles and alternating paths beginning and ending with edges in $\Omega^*$ at vertices in $B$.  We now consider three cases:

\begin{enumerate}
\item There exists an alternating cycle $C$, containing some edge of the form $\{(l_u,r_u) \}$: Let $C' = C \cap E'$ and $C^* = C \cap E^*$. Since the cycle is alternating, removing edges $C'$ from $\hat{G}^1$ and adding edges in $C^*$ to create a new graph $G$ is possible and preserves the degree sequence.  Further, since the graph of simple graphs is connected, there exists a sequence of double edge swaps to create $G$ from $\hat{G}^1$.  However, $G$ still contains edges $(u,l_u)$ and $(u,r_u)$ as these edges were precluded from set $\Omega'$, but since $(l_u,r_u)$ was in $C'$ it is not in $G$ and thus $(u,l_u)$ and $(u,r_u)$ form an open wedge $l_uur_u$ contradicting the maximality of $\hat{G}^1$.
\item There is an alternating path $L$ beginning and ending with edges in $\Omega^*$ at nodes $u,v\in B$ where $u\not = v$: Since $B\subseteq V^1$, lemma \ref{V1clique} grants that $(u,v)\in E'$ and thus not also in $\Omega^*$.  The union $(u,v) \cup L$ produces a cycle with edges alternatingly in  $E'$ and not in $E'$ and, as in the first case, augmenting $\hat{G}^1$ with this cycle produces a graph without a edge $(u,v)$, in violation of lemma \ref{V1clique} (Note, by classification $G^1$ cannot contain any edges in $V^0$).
\item There is an alternating path $L_l$ beginning and ending at the same vertex $l_u\in B$ and with edges in $\Omega^*$: Since $r_u$ has a lower degree in $G*$ than in $G'$, there must be some alternating path $L_r$ beginning at $r_u$.  Further, if the second case doesn't hold, then neither $L_l$ nor $L_r$ can visit any other vertex in $B$ other than $l_u$ and $r_u$ respectively.  Let $r_1$ be the first vertex in path $L_r$.  Since $(r_1,r_u) \in \Omega^*$, then by lemma \ref{V1clique} $r_1\in V^2$.  Next, if $(l_u,r_1)\in E'$, then notice that the union $(l_u,r_1) \cup L_l \cup (l_u,r_u) \cup (r_u,r_1)$ creates an alternating cycle that includes $(l_u,r_u)$, as in the first case.  If $(l_u,r_1)\not \in E'$ then the union of $(r_u,u)\cup (u,r_1)\cup L_r\setminus(r_u,r_1)$ creates an alternating cycle, and augmenting $\hat{G}^1$ with this cycle would create open wedge $l_uur_1$
\end{enumerate}

\end{proof}

\subsection{The structure of  $\hat{G}^2$}
A similar alternating path argument can be applied to $\hat{G}^2$, but in some ways it's easier to investigate $\hat{G}^2$ directly.  

First, notice that in a graph $\hat{G}^2$, $V^1$ is nonempty, since a vertex $u\in V^0$ with $k_u \ge 2$ must have two neighbors but since $V^0$ does not contain a triangle only one of $u$'s neighbors can be in $V^0$. {Since it is also possible that there are multiple pairs of connected vertices in $V^0$, we will let $K^0$ refer to any pair of connected vertices in $V^0$. We will show that each such $K^0$ shares the same connections into $V^1$. First though:}

\begin{lemma}
For $\hat{G}^d$, $d\ge 2$, for $u\in K^0$ and any $x\in V^0$ then $k_u\ge k_x$.  \label{d2V0degrees}
\end{lemma}
\begin{proof}
Suppose to the contrary that there exists $x\in V^0$ with $k_x > k_u$.   Since $u\in K^0$ there exists $v\in N(u)\cap V^0$.  Lemma \ref{K0_edges} implies that $x\not\in N(u)\cup N(v)$ {as otherwise $u$ must have the same neighbors as $x$ contradicting that $k_x > k_u$.  Further, since $k_x>k_u$} there exists  $y,z\in N(x)$ with $y,z\not \in N(u)$. {Lemma \ref{K0_edges} again implies that $y\not \in N(v)$, as otherwise $y\in N(u)$. } Notice that swap $(x,y),(u,v) \leadsto (x,u),(v,y)$ creates open wedge $zxu$, a contradiction.
\end{proof}
\noindent {
Since, the definition of $\hat{G}^2$ requires that there is some $u\in V^0$ with $k_u\ge 2$, lemma \ref{d2V0degrees} also implies that  $k_u\ge 2$ for $u\in K^0$.
 Consider the following abbreviations:
\begin{definition}[$V^1_u$, $V^1_K$ and $n_k$]
\label{Vsubu}
In a graph $\hat{G}_i$, and $u\in V^0$, let $V^1_u = N(u)\cap V^1$.  For any clique $K$ in $V^0$, let $V^1_K = \bigcup_{u\in K} V^1_u$. Finally denote $n_k = |V^1_K|$.
\end{definition}
}


\begin{lemma}
For $\hat{G}^d$, $d\ge 2$, $u\in K^0$ and any $x$ then either $ V^1_u \subseteq N(x)$ or $N(x) \subseteq V^1_u$.  \label{d2edgesIn}
\end{lemma}
\begin{proof}
{Consider the contradiction: that $V^1_u \not \subseteq N(x)$ because there exists $w\in V^1_u$, $w\not \in N(x)$ and $N(x) \not \subseteq V^1_u$ because there exists $y\in N(x)$, $y\not \in V^1_u$.} $x$ is not connected to $u$ as otherwise there exists open wedge $xuw$.   If $y\in N(u)$, then it must be that $y\in V^0$ (otherwise $y\in V^1_u$) and thus there exists open wedge $xyu$.  Thus $y\not \in N(u)$. 

As {$d\ge2$}, there exists $v\in N(u)\cap K^0$.  Since $y,x\not \in N(u)$ then by lemma \ref{K0_edges} $y,x\not \in N(v)$.  Now notice that swap $(x,y),(u,v)\leadsto(u,x),(v,y)$ creates open wedge $xuw$.
\end{proof}

Notice that this also gives that $V^3 = \emptyset$, that $k_y\ge k_x$ for any $y\in V^1_K$ and any $x$ and, in conjunction with lemma \ref{K0_edges}, that $V^1_u = V^1_K$ for any $u\in K^0$. 
Together {lemmas} \ref{d2edgesIn}, {and} \ref{d2V0degrees} give that all vertices $u\in V^0$ have at most one neighbor outside of $V^1_K$ { (for any $K$}), which is the key to the following lemma.  

\begin{lemma}
Every $m$-loopy $\hat{G}^2$, is $m^*$-loopy. \label{k2graph}
\end{lemma}
\begin{proof}
We will show that any $m$-loopy $\hat{G}^2$ is $m^*$-loopy by showing that the simplified degree sequence of $\hat{G}^2$ become a non-simple-graphical if a self-loop is added to any vertex in $V^0$. The proof will be based off the following principle: 
 For any simple graph and subset $U\subseteq V$, the sum of the degrees in $U$ can be no larger than $2$ times the possible number of edges internal to $U$ plus the number of edges coming into $U$.  Since the number of edges internal to $U$ is less than $\frac{1}{2}|U|(|U|-1)$ and an external vertex $w$ can contribute at most $\mathrm{min}(k_w,|U|)$ edges, then:\footnote{Notice that this is the necessary condition and the easier proof direction of the Erd\H{o}s-Gallai theorem.
 \begin{equation}
 \label{ErdosG_general}
 \sum_{u \in U} k_u \le |U|(|U|-1) + \sum_{w\not \in U} \mathrm{min}(k_w,|U|).
 \end{equation}}
We will let $U=V^1_K$, and show that for any given $m$-loopy $\hat{G}^2$ the number of edges are tight with the available degree, implying that adding any number of additional self-loops results in non-simple-graphical degree sequences.  
First, since each node in $V^1_K$ has a self-loop, the total available degree for edges in $V^1_K$ is $\sum_{u\in V^1_K} k_u -2=\sum_{u\in V^1_K}\bar{k}_u$. Next consider the edges with endpoints in $V^1_K$. 
Since $ V^1_K\subseteq V^1$, it is a clique and thus contributes $n_k(n_k-1)$ to the degrees inside $V^1_K$.  Lemmas \ref{K0_edges} and \ref{d2edgesIn} implies that for all $u\in K^0$, $u$ connects to all of {$V^1_K$}.  Further, lemma \ref{d2edgesIn} gives that for any remaining vertex $x$, either $N(x) \subseteq V^1_u$, in which case all $k_x$ edges from $x$ connect to $V^1_K$, or $V^1_K\subseteq N(x)$ in which case $x$ connects to all of $V^1_K$.   
 Aggregating these leads to the statement: 
\begin{equation}
\sum_{u\in V^1_K} \bar{k}_u = n_k(n_k-1)+\sum_{u\in K^0} (\bar{k}_u-1) + \sum_{u \in V\setminus (K^0\cup V^1_K)} \min(n_k,\bar{k}_u),  \label{ErdosG}
\end{equation}
where $\bar{k}_u$ is the simplified degree sequence (definition \ref{def:simplifiedegreesequence}) and $\bar{k_u}-1=k_u-1=n_k$ for $u\in K^0$.
Notice if any subset of vertices $S\in V^0$ have their degree reduced by $2$ then each vertex $u\in S$ connects to at least one less vertex in $V^1_K$.  Thus, reducing the degree of any vertex in $V^0$ by $2$ reduces the right side of equation \ref{ErdosG}, but not the left side, thereby inverting the inequality required by equation \ref{ErdosG_general}. Thus, 
the new degree sequence would not be simple-graphical implying that $\hat{G}^2$ is $m^*$-loopy. 
\end{proof}

Thus we have shown:
\begin{lemma}
A $m$-loopy graph $\hat{G}^d$, $d\le 2$, is $m^*$ loopy.  \label{masterLoopy}
\end{lemma}
However, as seen in Figure \ref{sampleGraphs} there exists $\hat{G}^d$, $d\ge 3$ which are not $m^*$-loopy.

\subsection{Structure of $\hat{G}^d$, $d\ge 3$}

Finally we investigate the possible structures of $\hat{G}^d$ for $d\ge 3$, showing that any $\hat{G}^3$ is in the class $Q_2$, any $\hat{G}^d$ for $d> 3$ is in the class $Q_1$ and thus $\hat{G}^d$, $d\ge 3$ indicates a disconnected $\mathcal{G}$.  First, we show the following:

\begin{lemma}
For $\hat{G}^d$, when $d\ge3$ all edges in $V^0$ are contained in $K^0$. \label{d3noV0edges}
\end{lemma}
\begin{proof}
Suppose not, then there exists $(x,y)\in V^0$ {where $(x,y)\not \in K^0$}. Since $(x,y) \not \in K^0$ then there is some $u\in K^0$ with $u\not \in N(x)$, and since $V^0$ is wedge free it must be that $x,y$ are disjoint from $K^0$.  Let $u,v,w\in K^0$, then swap $(x,y),(u,v)\leadsto (x,u),(y,v)$ creates wedge $xuw$.
\end{proof}

\begin{lemma}
{For $\hat{G}^d$ with $d\ge 3$, and any edge $(x,y)$, $x\ne y$ then either $x$ or $y$ must be in $V^0\cup V^1$.  If $d>3$, the above result holds even if $x=y$}
\label{d3onlyV1V2}
\end{lemma}
\begin{proof}
Suppose not, then there exists $(x,y)$ with both $x,y \not \in V^0\cup V^1$.  For $u,v,w\in K^0$ notice that swapping $(u,v),(x,y)\leadsto(u,x),(v,y)$ creates open wedges $xuw$ and $yvw$. {If $x\not =y$ then closing either $xuw$ or $yvw$ creates a graph with an additional self-loop.}
If $x=y$ and there are $u,v,w,z\in K^0$ then {the additional swaps $(x,u)(u,w)\leadsto (u,u)(x,w)$ and $(x,w)(w,z)\leadsto (w,w)(x,z)  $ create two new self-loops, compensating for the loss of the self-loop $(x,x)$}.  
\end{proof}

\begin{lemma}
For $\hat{G}^3$, either $V^\infty$ or $V^1$ is empty.  \label{d3V1Vinf}
\end{lemma}
\begin{proof}
Suppose not, then there exists $x\in V^\infty$ and $y\in V^1$. Let $u,v,w\in K^0$.  Lemma \ref{K0V1Conn} implies that $u,v,w\in N(y)$.  Consider the following swaps: $(x,x)(u,w)\leadsto (u,x)(w,x)$, then $(y,u)(u,x)\leadsto (u,u)(y,x)$ and $(v,w)(w,x)\leadsto(w,w)(v,x)$ which net creates a self-loop.  
\end{proof}

\begin{lemma}
For {$m$-loopy} $\hat{G}^d$, when $d=3$, all vertices in $V^2$ have degree $|V^1|+2$, while when $d > 3$, $V^2=\emptyset$  \label{d3V1degrees}.
\end{lemma}
\begin{proof}
{Consider a vertex $u\in V^2$ and $v\in V^0$.}
	Since $V^3=\emptyset$ and $V^2$ has no internal edges by lemma \ref{d3onlyV1V2}, then ${k_u\le} |V^1|+2$, while ${k_v=} |V{^1}|+|K^0|-1 = |V^1|+d-1$. 
{
Since $\hat{G}^d$ is $m$-loopy, then $k_u\ge k_v$, implying that $|V^1|+2\ge k_u \ge k_v = |V^1|+d-1$.  Thus,   
}
$d=3$, vertices in $V^2$ have degree $|V^1|+2$, and when $d>3$,  $V^2$ must be empty.
\end{proof}

Taken together, lemmas \ref{d3V1degrees} and \ref{d3noV0edges}  imply that $\hat{G}^d$ for $d\ge 3$ is composed of a single large clique on $K^0\cup V^1$ along with vertices that solely connect into that clique.  Further, all vertices in $V^0$ have less than or equal degree than all other vertices.  Meanwhile, the degrees of vertices in $V^2$, and  $K^0$ are $|V^1|+2$ and $|V^1|+d-2$ respectively while those in $V^1$ and $V^0\setminus K^0$ have lower and upper bounds $|V^1| +|V^2| +d +1$ and $|V^1|$ respectively. Taken together, these constraints on the form of $\hat{G}^d$, $d\ge3$ (as summarized in Figure \ref{GdForm}), can be used to detect degree sequences for which $\mathcal{G}$ is disconnected.

\begin{theorem}
A graph $\hat{G}^3$ is in the class $Q_2$ . \label{G3inQ}
\end{theorem}
\begin{proof}
The first criterion in the definition of the class $Q_2$, that there exists $u$ such that $|N(u)\cap V^0|=2$,  is satisfied by the definition of $\hat{G}^3$.
Lemma \ref{d3noV0edges} establishes that there aren't edges in $V^0$ outside of $K^0$, and lemma  \ref{K0V1Conn} gives that any node in $K^0$ connects to all of $V^1$; together these satisfy the second criterion.  Lemma \ref{V1clique} shows hat $V^1$ is a clique, the third criterion of $Q_2$. Lemmas \ref{d3onlyV1V2} and \ref{d3V1degrees} imply that each $u\in V^2$ has $N(u) = V^1$, the fourth criterion of $Q_2$.  The last two criteria of $Q_2$ are satisfied by lemmas \ref{d3V1Vinf} and \ref{d3onlyV1V2}.  Thus, $\hat{G}^3\in Q_2$.
\end{proof}
\begin{theorem}
\label{GdinQ}
A graph $\hat{G}^d$ for $d> 3$ is in the class $Q_1$.
\end{theorem}
\begin{proof}
The first criterion in the definition of the class $Q_1$, the existence of  a clique inside $V^0$ is satisfied by the definition $\hat{G}^d$ for $d> 3$.  The second criterion of $Q_1$ is given by Lemma \ref{d3noV0edges}. The third criterion for $Q_1$, that $V^1\cup K^0$ is a clique, is given by lemmas \ref{V1clique} and \ref{K0V1Conn}. Finally, lemmas \ref{d3V1degrees}  and \ref{d3onlyV1V2} imply that in $\hat{G}^d$ for $d> 3$, $V^2=V^\infty = \emptyset$, the fourth criterion of $Q_1$. Thus, $\hat{G}^d\in Q_1$ for $d> 3$.
\end{proof}

An immediate consequence of {these two theorems}, along with lemmas \ref{masterLoopy} and \ref{lem:01mstar} and corollary \ref{GQ1Q2} is the following:

\begin{theorem}
A degree sequence has a disconnected $\mathcal{G}$ if and only if there is some graph in $Q_1$ or $Q_2$ in $\mathcal{G}$.  \label{Q1Q2Exact}
\end{theorem}

\begin{corollary}
Aside from the degree sequences associated with the cycle and the clique all simple graphical degree sequences have a connected $\mathcal{G}$.
\end{corollary}
\begin{proof}
Applying the Erd\H{o}s-Gallai theorem to the set $V^1$ in a graph in $Q_1$ or $Q_2$ reveals that such a graph's degree sequence is not simple-graphical, unless $|V^1|=0$, in which case the graph is either a clique, or has degree sequence $\{2,2,2,....,2 \}$. 
\end{proof}

Thus, many of the most commonly examined degree sequences have a connected $\mathcal{G}$.  However, in the space of loopy-graphs, there are many possible degree sequences which are loopy-graphical, but not simple-graphical (for example, those in Figure \ref{sampleGraphs}).  In this next section, we discuss several ways to detect if a loopy-graphical degree sequence has a connected or disconnected space.

\section{Detecting connectivity in $\mathcal{G}$} \label{detecting}

For many applications, detecting if a degree sequence is not at risk of being disconnected can be achieved simply by examining the maximum degree.  Let $n^*$ be the number of  nodes with nonzero degree in a degree sequence $\{k_i \}$.

\begin{theorem}
For degree sequence $\{k_i \}\ne\{2,2,2,...,2 \}$, if $\max_i k_i < 2\sqrt{n^*-3}+1$ then $\mathcal{G}(\{ k_i \})$ is connected. \label{maxDegTest}
\end{theorem}
\begin{proof}
Since only degree sequences that can create a graphs in $Q_1$ and $Q_2$ have a disconnected graph of graph, we need only show that the maximum degree of graphs in  $Q_1$ and $Q_2$ is never less than  {$2\sqrt{n^*-3}+1$}.  For a graph $G\in Q_1 \cup Q_2$, let $\alpha = |V^1|$.  Notice that the highest degree node in $G$ must be in $V^1$. Counting the edges into $V^1$: at least three nodes in $K^0$ connect to all nodes in $V^1$ and the remaining $n^*-3-\alpha$ nodes have at least one edge into $V^1$.  Since $V^1$ is a clique, there are at least $\alpha(\alpha-1)+3\alpha+(n^*-3-\alpha)$ edge endpoints into $V^1$, and thus the maximum degree of a node in $V^1$ must be at least $\alpha+1+\frac{n^*-3}{\alpha}$.  Minimizing this over $\alpha$ yields the bound $ 2\sqrt{n^*-3}+1$
\end{proof}

When the maximum degree is larger than the bound in theorem \ref{maxDegTest}, { the following procedure can} exactly identify all degree sequences {that can create a $Q_1$ or a $Q_2$ other than $\{k_i\} = \{2,2,2,...,2 \}$ and $\{k_i\} =\{n-1,n-1,...,n-1 \}$.  Namely, the procedure either creates a $Q_1$ or $Q_2$ graph or terminates.  In the following, let $n_a$ and $n_b$ denote the number of vertices which have remaining degree equal to $a$ and $b$ respectively.}

\begin{enumerate}
\item If all remaining {non-zero} degrees have {exactly} two different values {$b> a\ge 3$}, $n_a\ge 3$, $b-2 = n_a+n_b-1$ and $a-2 = n_b$ then it is possible to place nodes with degree $b$ in $V^1$, three vertices with degree $a$ into $K^0$ and the remaining vertices with degree $a$ as vertices in $V^2$, creating a graph {$Q_2$}. 
\item If all remaining {non-zero} degrees have {exactly} two different values {$b>a\ge 3$}, $n_a\ge 3$,  $a = b-2$ and  $a = n_a+n_b-1$ then it is possible to create a clique with self-loops at each vertex of degree $b$, creating a graph {$Q_1$}.
\item {Let $u = \mathrm{argmin}_{i} \{k_i| k_i\not=0\}$}
\item {Connect $u$ to the $k_u$ largest degree vertices}
\item Reduce the largest $k_i$ degrees by $1${, and set $k_u=0$}
\item {Return to step 1.}
\end{enumerate}

To illustrate this procedure, consider the following example of degree sequence ${\{8,3,3,3,3,3,1\}}$ (the example $Q_2$ graph seen in figure \ref{FigQ1Q2} has this degree sequence).  As the degree sequence contains more than two different values, the first pass of the algorithm skips steps 1 and 2 and connects the degree $1$ vertex to the degree $8$ vertex, resulting in degree sequence $\{7,3,3,3,3,3,0\}$.  The degree sequence now only has two non-zero values, $7$ and $3$, so that in steps 1 and 2, $a=3$, $n_a=5$ and $b=7$ with $n_b=1$. The corresponding checks in step 1, check that $n_a = 5\ge 3$, that $b-2 = 5= n_a+n_b-1$, and that $a-2=1=n_b$.  Since the degree sequence passes each of these checks, it's possible to create a graph $Q_2$.  
Applying the procedure to the degree sequence $\{9,8,5,5,5,5,2,1 \}$ would first connect the degree $1$ vertex to the degree $9$ vertex and subsequently connect the degree the $2$ vertex to the first two vertices resulting in a new degree sequence $\{7,7,5,5,5,5,0,0 \}$.  Since this new degree sequence has only two non-zero values the procedure halts on the third iteration and creates the graph displayed as an example $Q_1$ graph in figure \ref{FigQ1Q2}.

The validity of this procedure is established below, while we provide full pseudo-code as Algorithm \ref{alg1} in the appendix.

\begin{theorem}\label{algCorrect}
For any degree sequence, the above procedure correctly identifies whether $\mathcal{G}$ is connected or disconnected. 
\end{theorem}
\begin{proof}

 { Theorem \ref{Q1Q2Exact} implies that $\mathcal{G}(\{k_u \})$ is disconnected if and only if $\{k_u \}$ can construct a graph in $Q_1$ or $Q_2$. It will follow that correctness simply requires understanding the structure of $Q_1$ and $Q_2$ presented in figure \ref{FigQ1Q2} and attempting to naively construct such a graph.  If the construction succeeds, then clearly $\{k_u \}$ can create a $Q_1$ or $Q_2$, and if it fails, then we will show that no such construction is possible. 
}

To see the correctness of this procedure, we will simply run through the logic in reverse order.  
Suppose that $G$ is any $Q_1$ or $Q_2$ with degree sequence $\{k_u\}$, we will show the procedure produces a graph connected to $G$ by trivial edge-swaps. 
 As given by the definition of $Q_1$ or $Q_2$, and as apparent in figure \ref{FigQ1Q2}, if all vertices not contained in $V^1$, $K^0$ or $V^2$ and their attached edges are deleted from $G$, the remaining degree sequence has only two non-zero values, which we denote as $b>a\ge 3$.    Let $n_a$ and $n_b$ denote the number of vertices with degrees $a$ and $b$.  Consider these two cases corresponding to steps 1 and 2:
 \begin{enumerate}
\item If $G$ was a $Q_1$ then $V^1\cap K^0$ form a clique (condition 3 of $Q_1$).  Counting the number of edges incident to nodes inside $V^1$ and $K^0$ reveals that  vertices of degree $b$ compose $V^1$, those of degree $a$ compose $K^0$, $n_a\ge 3$, $b=a+2$ and $a =  n_a+n_b-1$. 
\item If $G$ was a $Q_2$ then $V^1$ is composed of vertices of degree $b$ and all remaining vertices in $K^0$ and in $V^2$ have degree $a$ (conditions 1-4 of $Q_2$).  It would thus follow from these same properties that $n_a\ge 3$, that each node in $V^1$ connects to all remaining nodes: $b-2 = n_a+n_b-1$, and that each node in $V^2$ or $K^0$ neighbors all of $V^1$ and has either two other neighbors or a self-loop $a-2 = n_b$.
\end{enumerate}
   Thus, steps 1 and 2 of the above procedure exactly identify a $Q_1$ and $Q_2$ which have had all vertices not contained in $V^1$, $K^0$ or $V^2$ deleted.

Thus, the only challenge is to identify which vertices are in $V^0\setminus K^0$, so that deleting these nodes allows for the tests in the first and second steps of the procedure.  Thankfully, this is not hard. The definition of $Q_1$ and $Q_2$ ensure that the vertices in $V^0 \setminus K^0$ can only connect into $V^1$, which guarantee that these vertices have the smallest degrees.  Similarly, the vertices in $V^1$ of a $Q_1$ and $Q_2$ graph must have a degree at least two higher than vertices in $V^0$ or $V_2$ (see figure \ref{FigQ1Q2}).  It follows immediately that if there are vertices in $V^0\setminus K^0$, then connecting the smallest degree vertex $u$, to the $k_u$ largest vertices, connects a vertex in $ V^0\setminus K^0$ to $k_u$ vertices in $V^1$.  

Since vertices in $V^0\setminus K^0$ have strictly smaller degree than those in $K^0$ then so long as there remain vertices in $V^0\setminus K^0$, the degree sequence has more than $2$ unique values.  Thus, repeating the procedure would eventually delete all vertices in $V^0\setminus K^0$.

The final concern to address is the possibility that while this procedure can create a $Q_1$ or $Q_2$ the particular choices of which vertices in $V^1$ connect to which vertices in $V^0\setminus K^0$ may cause the procedure to miss some $Q_1$ or $Q_2$.  We alleviate this concern by showing that if a degree sequence can construct a $Q_1$ or $Q_2$ then there are a sequence or edge swaps that could also create one with the same connections from $V^1$ to $V^0\setminus K^0$ as the procedure creates.
Notice that for two edges $(u,v)$ and $(x,y)$ with $u,x\in V^0\setminus K^0$ and $v,y\in V^1$, swap $(u,v),(x,y)\leadsto (u,y),(x,v)$ exchanges $u$ and $x$'s neighbors in $V^1$.  Thus, if a \joel{degree sequence} $\{k_u \}$ can create a $Q_1$ or $Q_2$ graph, then it can also create a $Q_1$ or $Q_2$ graph with any permutation of the edges from $V^0\setminus K^0$ to $V^1$, such as the one which was just constructed.  Thus this algorithm constructs a $Q_1$ or $Q_2$ if it is possible.  If it is not possible to construct a $Q_1$ or $Q_2$ then the procedure will eventually delete every vertex, or step 4 will fail because there will be fewer than $k_u$ non-zero degree vertices remaining.  

 \end{proof}

\section{Sampling loopy-graphs}

A small change to $\mathcal{G}$ can connect the space. For distinct $u,v,w$ consider the following triple edge swap, the `triangle-loop' swap, $(u,v),(v,w),(w,u)\leadsto(u,u),(v,v),(w,w)$ along with its reverse $(u,u),(v,v),(w,w)\leadsto(u,v),(v,w),(w,u)$.  Let $\mathcal{G}_\triangle$ be the graph $\mathcal{G}$ but with additional edges connecting graphs which are separated by a single triangle-loop swap {along with an additional number of self-loops in $\mathcal{G}_\triangle$ to preserve detailed balance (as discussed below)}.

\begin{theorem}
$\mathcal{G}_\triangle$ is connected. \label{tripleSwaps}
\end{theorem}
\begin{proof}
Since every $\hat{G}^d$, $d\ge3$ contains a triangle in $V^0$, no such graph has the maximal number of self-loops. Thus for every degree sequence, any graph is connected to a $\hat{G}^d$ for $d\le 2$, which are $m^*$-loopy and since all $m^*$-loopy graphs are connected, the space is thus connected.
\end{proof}

This allows for an MCMC sampler of the uniform distribution of graphs in $\mathcal{G}_\triangle$.  For a given degree sequence, if Algorithm \ref{alg1} indicates that $\mathcal{G}$ is connected then the standard double edge swap MCMC in \cite{MRCconfiguration} suffices. On the other hand, if Algorithm \ref{alg1} returns a valid $\mathcal{G}^d$ for $d\ge 3$ then triangle-loop swaps are required to connected the space, as in Algorithm \ref{alg2}.  Stated succinctly the stub-labeled version does the following:   

From any graph $G$: with probability $\epsilon>0$ pick three edges from $G$, if possible perform a triangle-loop swap, otherwise resample $G$; with probability $1-\epsilon$ pick $2$ edges at random, if possible perform a double edge swap, otherwise resample $G$. 
  
  \begin{algorithm}
\caption{MCMC step (labeled stubs) }
\label{alg2}
\begin{algorithmic}
\REQUIRE loopy-graph $G$, stubs\_labels $\in [True,False]$
\ENSURE a loopy-graph adjacent to $G$ in $\mathcal{G}_\triangle$
\IF {$Unif(0,1)<  \epsilon$} 
	\STATE choose three edges at random
	\IF{edges create a triangle or are self-loops \& triangle-loop wouldn't create multiedges} 
		\STATE perform triangle-loop swap
	\ENDIF
\ELSE 
	\STATE choose two edges $e_1$ and $e_2$ at random
	\IF {double edge swap wouldn't create multiedges} 
		\IF{neither $e_1$ or $e_2$ is a self-loop OR not stubs\_labels} 
			\STATE perform double edge swap
		\ELSE 
			\IF {$Unif(0,1)<  \frac{1}{2}$}
				\STATE perform double edge swap
			\ENDIF
		\ENDIF
	\ENDIF
\ENDIF
\State \textbf{return } $G$
\end{algorithmic}
\end{algorithm}

  For any $\epsilon>0$, theorem \ref{tripleSwaps} gives that this procedure will be able to reach all graphs in $\mathcal{V}$, however, since the majority of proposed triple swaps will not result in a new graph, the value of $\epsilon$ that produces the optimal mixing time is likely small. 
   In order to see that triangle-loop swaps preserve {
detailed balance ($P(G,G') = P(G',G)$)
}
notice that since each triangle-loop swap is reversible {and} that at any graph $G$ each of the exactly ${m}\choose{3}$ sets of three distinct edges corresponds to an incoming edge, either from a {graph-of-graph} self-loop, or a valid triangle-loop swap.
  
 To see that $\mathcal{G}_\triangle$  is aperiodic consider several cases. Notice that for any degree sequence, if $|\mathcal{V}|\ge 2$, $\mathcal{G}_\triangle$  must contain a graph with at least one of the following: a triangle, an open wedge, two self-loops or two independent edges. Attempting to \joel{swap} two sides of a triangle or two self-loops would create a \joel{multiedge}, and this attempted swap corresponds to a self-loop in $\mathcal{G}_\triangle$, which implies that $\mathcal{G}_\triangle$  is aperiodic.  Any graph with an open wedge or two independent edges has a sequence of three double-edge swaps which return to the same graph, this combined with the reversible nature of double-edge swaps implies that $\mathcal{G}_\triangle$  is aperiodic.

This leads to the following theorem, which lets us conclude that Algorithm \ref{alg2} forms the basis for a  MCMC sampler of loopy-graphs.
\begin{theorem}
A random walk on $\mathcal{G}_\triangle$ has a uniform stationary distribution.
\end{theorem}
\begin{proof}
As an aperiodic, connected graph $\mathcal{G}_\triangle$, {with transitions that obey detailed balance, a random walk} has a unique uniform stationary distribution. 
\end{proof}

\section{Conclusion}

By examining the possible structures of graphs with the maximum number of self-loops reachable via double edge swaps we have a complete categorization of the degree sequences where double edge swaps can change any graph into any other valid graph. This understanding is exemplified in Algorithm \ref{alg1}, which can detect whether a degree sequence has a connected space or not. Further, we proved that augmenting double-edge swaps with triangle-loop swaps connects the space of loopy-graphs, creating the first provably correct MCMC technique for sampling loopy-graphs. In addition to filling a gap in the understanding of graph space connectivity, this work builds a tool to allow for the sampling of loopy-graphs and their subsequent use as statistical null-models. As greater emphasis is placed on sampling graphs without labeled-stubs the need for carefully sampling loopy-graphs will likely increase. 

\section{Acknowledgments}

This work was aided by comments and suggestions from: Bailey Fosdick, Dan Larremore and Johan Ugander.

\section{Appendix}

{In this appendix we formalize the procedure for determining whether a degree sequence has a connected $\mathcal{G}$.}

\begin{algorithm}
\caption{Attempt non $m^*$-loopy wiring}
\label{alg1}
\begin{multicols}{2}
\begin{algorithmic}[1]
\REQUIRE{degree sequence $\{ k_i\}$}
\ENSURE{a non $m^*$-loopy graph $\hat{G}$, otherwise $False$}
\STATE $n = |\{ k_i | k_i>0 \}|$ 
\STATE $G  \leftarrow$ graph initialized with vertices from $\{ k_i\}$ 
\STATE sort $\{ k_i\}$ in decreasing order 
\IF{$n \le 2$}  
\State \textbf{return } $False$
\ENDIF
\IF{$\min_i k_i = \max_i k_i = 2$}  
\State \textbf{return }{a cycle graph on $n$ vertices}
\ENDIF
\IF{$\min_i k_i = \max_i k_i = n-1$}  
\State \textbf{return }{a clique on $n$ vertices}
\ENDIF
\FOR {$j \in  0:n $} 
\IF{$\{ k_i\}$ has exactly two unique values} 
\STATE $a \leftarrow \min_i k_i$
\STATE $b \leftarrow \max_i k_i$
\STATE $n_a \leftarrow$ number of occurrences of $a$ in $\{ k_i\}$
\STATE $n_b \leftarrow$ number of occurrences of $b$ in $\{ k_i\}$
\STATE $n_t \leftarrow n_a+n_b$
\IF{$a\ge3$ and $n_a\ge3$}
\IF{ $a = b-2$ and $a = n_t-1$}
%
\FOR{$u\in 0:n_b$}
\STATE add edge $(u,u)$ to $G$ 
\ENDFOR
\STATE add clique on vertices $0:n_t$ to $G$ 
\State \textbf{return } $G$
\ENDIF 
\IF{ $b-2 = n_t-1$ and $a-2 = n_b$}
\FOR{$u\in 0:n_b$}
\FOR{$v\in 0:n_t$}
\STATE add edge $(u,v)$ to $G$ 
\ENDFOR
\ENDFOR
\FOR{$u\in 0:(n_t-3)$}
\STATE add edge $(u,u)$ to $G$ 
\ENDFOR
\STATE create clique on vertices $(n_t-3):n_t$ 
\State \textbf{return } $G$
\ENDIF
\ENDIF
\ENDIF
\STATE $MinInd \leftarrow {n-j-1}$
\STATE $MinDeg \leftarrow k_{MinInd}$
\STATE delete($k_{MinInd}$)
\IF{$MinDeg>  |\{ k_i\}|$} 
\State \textbf{return }{ False} 
\ENDIF
\FOR {$y \in 0:MinDeg$}
\STATE $k_y \leftarrow k_y -1$
\STATE add edge $(MinInd,y)$ to $G$
\ENDFOR
\STATE sort $\{ k_i\}$ in decreasing order
\ENDFOR
\State \textbf{return } False
\end{algorithmic}
\end{multicols}
\end{algorithm}

\bibliography{referencesJN}{}
\bibliographystyle{plain}

\end{document}